\documentclass{amsart}

\usepackage{mathrsfs}

\usepackage[dvips]{graphicx}
\usepackage{amsmath}
\usepackage[latin1]{inputenc}
\usepackage{amsfonts}
\usepackage{amssymb}
\usepackage{graphicx,epsfig}
\usepackage{amsthm}

\usepackage{amscd}
\usepackage{dsfont}
\usepackage[all,dvips]{xy}
\usepackage{fancyhdr}
\usepackage[T1]{fontenc}

\input cyracc.def

\allowdisplaybreaks

\begin{document}

\title{Endomorphism Algebras and q-Traces}

\author{Run-Qiang Jian}

\address{D\'{e}partement de Math\'{e}matiques, Universit\'{e} Paris
Diderot (Paris 7), 175, rue du Chevaleret, 75013, Paris, France}
\email{jian@math.jussieu.fr}

\address{Department of Mathematics, Sun Yat-sen University,
135, Xingang Xi Road, 510275, Guangzhou, P. R. China}

\thanks{}


\date{}


\maketitle

\begin{abstract}
For a braided vector space $(V,\sigma)$ with braiding $\sigma$ of
Hecke type, we introduce three associative algebra structures on
the space $\oplus_{p=0}^{M}\mathrm{End}S_\sigma^p(V)$ of graded
endomorphisms of the quantum symmetric algebra $S_\sigma(V)$. We
use the second product to construct a new trace. This trace is an
algebra morphism with respect to the third product. In particular,
when $V$ is the fundamental representation of
$\mathcal{U}_{q}\mathfrak{sl}_{N+1}$ and $\sigma$ is the action of
the $R$-matrix, this trace is a scalar multiple of the quantum
trace of type $A$.
\end{abstract}

\newtheorem{theorem}{Theorem}[section]
\newtheorem{lemma}[theorem]{Lemma}
\newtheorem{proposition}[theorem]{Proposition}
\newtheorem{definition}[theorem]{Definition}
\newtheorem{corollary}[theorem]{Corollary}
\newtheorem{remark}[theorem]{Remark}

\section{Introduction}
More than twenty years ago, H. Osborn studied the space
$\oplus_{i\geq 0}\mathrm{End}\bigwedge^i (V)$ of graded
endomorphisms of the exterior algebra in order to give an
algebraic construction of Chern-Weil theory (\cite{O1, O2}). He
introduced three associative products on this space. The first one
is just the composition of endomorphisms. Since the exterior
algebra is also a coalgebra, he defined the second one to be the
convolution product. And then he combined the first two ones to
construct the third product. Assuming that $\dim\bigwedge^i (V)=1$
for sufficiently large $i$, he constructed a trace function by
using the second product. This trace gives the usual one when it
is restricted on $\mathrm{End}(V)$. And it is an algebra morphism
when one considers the third product.

On the other side, after the creation of quantum groups by
Drinfel'd \cite{D} and Jimbo \cite{J}, mathematicians use
Yang-Baxter operators to quantize various classical objects in
algebra and find many interesting phenomena. Since symmetric
algebras and exterior algebras are defined by using flips which
are trivial Yang-Baxter operators, it seems quite reasonable and
possible to quantize them. In his paper \cite{Gu}, Gurevich
studied Yang-Baxter operators of Hecke type, which he called Hecke
symmetries. And then he defined the symmetric algebra and the
exterior algebra with respect to these operators. They are
analogue to the usual ones. Later, different aspects of these
algebras were discussed in \cite{HH} and \cite{Wam}. In
\cite{Ro2}, a very remarkable property of the quantized symmetric
algebra was discovered. For some special Yang-Baxter operators,
the symmetric one, as Hopf algebra, is isomorphic to the "upper
triangular part" of the quantized enveloping algebra associated
with a symmetrizable Cartan matrix.

Naturally, it is interesting to see what will happen when one
extends Osborn's trace to the quantum case. Let $(V, \sigma)$ be a
braided vector space with braiding $\sigma$ of Hecke type, and
$S^p_\sigma(V)$ be the $p$-th component of the quantum symmetric
algebra $S_\sigma(V)$ built on $(V,\sigma)$. We assume that $\dim
S_\sigma^M(V)=1$ for some $M$ and $\dim S_\sigma^p(V)=0$ for
$p>M$. Then, on the vector space $\oplus_{p=0}^{M}
\mathrm{End}S_\sigma^p(V)$, the convolution product, the third
product and the trace can be constructed step-by-step following
the ones in \cite{O2}. And this trace, called q-trace,  is an
algebra morphism with respect to the third product. In particular,
let $V$ be the fundamental representation of
$\mathcal{U}_{q}\mathfrak{sl}_{N+1}$ and $\sigma$ the braiding
given by the $R$-matrix of $\mathcal{U}_{q}\mathfrak{sl}_{N+1}$.
Then $\sigma$ is of Hecke type and $S_\sigma^p(V)$ vanishes when
$p$ is sufficiently large. To our surprise, the q-trace in this
case has already existed for more than one decade.

In the theory of quantum groups, there is an important invariant
which generalizes the usual trace of endomorphisms. It is the
so-called quantum trace. Let $\mathcal{C}$ be a ribbon category
with unit $I$, $V$ be an object of $\mathcal{C}$ and $f$ be an
endomorphism of $V$. The quantum trace $\mathrm{tr}_{q}(f)$ of $f$
is an element in the monoid $\mathrm{End}(I)$ (see, e.g.,
\cite{Ka}). It coincides with the usual trace when
$\mathcal{C}=Vect(k)$. When we take $\mathcal{C}$ to be the
category of finite dimensional representations of $u_\epsilon$
(for the definition, one can see \cite{KRT}), the quantum trace is
given by composing the usual trace with the action of the
group-like elements $K_i$'s. This is a functorial approach. After
an easy computation, we can show that our q-trace is a scalar
multiple of the quantum trace. So we get a more elementary
approach to the quantum trace of type $A$.

This paper is organized as follows. In Section 2 we define the
three products on $\oplus_{p=0}^{M} \mathrm{End}S_\sigma^p(V)$ for
a braided vector space $(V,\sigma)$ with a braiding $\sigma$ of
Hecke type. Then we construct the q-trace of $\oplus_{p=0}^{M}
\mathrm{End}S_\sigma^p(V)$ and prove that it is an algebra
morphism with respect to the third product. In Section 3, we apply
our constructions to the special braided vector space
$(V,\sigma)$, where $V$ is the fundamental representation of
$\mathcal{U}_{q}\mathfrak{sl}_{N+1}$ and $\sigma$ is the braiding
given by the $R$-matrix of $\mathcal{U}_{q}\mathfrak{sl}_{N+1}$.
We prove that the q-trace for this case is just a scalar multiple
of the quantum trace of $\mathcal{U}_{q}\mathfrak{sl}_{N+1}$.

\section*{Notation}
We fix our ground field to be the complex number field
$\mathbb{C}$.

We denote by $\mathfrak{S}_{p}$ the symmetric group of
$\{1,\ldots,p\}$. For $\{i_{1},\ldots,i_{k}\}\subset
\{1,\ldots,p\}$, we denote
$l(i_{1},\ldots,i_{k})=\sharp\{(i_{s},i_{t})|1\leq s<t\leq
k,i_{s}>i_{t}\}$. And for any $w \in \mathfrak{S}_{p}$,
$l(w)=l(w(1),\ldots,w(p))$. It is just the length of $w$.

An $(i,j)$-shuffle is an element $w\in \mathfrak{S}_{i+j}$ such
that $w (1) < \cdots <w (i)$ and $w (i+1) < \cdots <w (i+j)$. We
denote by $\mathfrak{S}_{i,j}$ the set of all $(i,j)$-shuffles.

Let $V$ be a vector space. A braiding $\sigma$ on $V$ is an
invertible linear map in $\mathrm{End}(V\otimes V)$ satisfying the
\emph{quantum Yang-Baxter equation}:
\begin{equation*}(\sigma\otimes
\mathrm{id}_{V})(\mathrm{id}_{V}\otimes \sigma)(\sigma\otimes
\mathrm{id}_{V})=(\mathrm{id}_{V}\otimes \sigma)(\sigma\otimes
\mathrm{id}_{V})(\mathrm{id}_{V}\otimes \sigma).\end{equation*} A
braided vector space $(V,\sigma)$ is a vector space $V$ equipped
with a braiding $\sigma$. For any $p\in \mathbb{N}$ and $1\leq
i\leq p-1$, we denote by $\sigma_i$ the operator
$\mathrm{id}^{\otimes(i-1)}\otimes \sigma\otimes
\mathrm{id}^{\otimes(p-i-1)}\in \mathrm{End}(V^{\otimes n})$. For
any $w\in \mathfrak{S}_{p}$, we denote by $T_w$ the corresponding
lift of $w$ in the braid group $B_p$, defined as follows: if
$w=s_{i_1}\cdots s_{i_l}$ is any reduced expression of $w$, where
$s_{i}=(i,i+1)$, then $T_w=\sigma_{i_1}\cdots \sigma_{i_l}$. We
also use $T_w^\sigma$ to indicate the action of $\sigma$.

Let $q$ be a nonzero number in $\mathbb{C}$. For $q\neq 1$ and any
$n=0,1,2, \ldots$, we denote $(n)_{q}=(1-q^{n})/(1-q)$, and

\[(n)_{q}!= \left\{
\begin{array}{lll}
1,&& n=0,\\
\frac{(1-q)\cdots (1-q^{n})}{(1-q)^{n}},&&n\geq 1.
\end{array} \right.
\]

\section{The q-trace}
In this section, we define the q-trace and prove that it is an
algebra morphism with respect to the third product. We start by
recalling some notions and properties of braidings of Hecke type
and quantum symmetric algebras for the later use. For more
details, one can see \cite{FG}, \cite{Gu}, \cite{Ro2} and
\cite{Wam}.
\subsection{Braidings of Hecke type and quantum symmetric algebras}
Let $(V,\sigma)$ be a braided vector space. The braiding $\sigma$
is said to be of \emph{Hecke type} if it satisfies the following
\emph{Iwahori's quadratic
equation}:\begin{equation*}(\sigma+\mathrm{id}_{V\otimes
V})(\sigma-\nu\mathrm{id}_{V\otimes V})=0,\end{equation*}where
$\nu$ is a nonzero scalar in $\mathbb{C}$.

In the rest of this section, $\sigma$ is always a braiding of
Hecke type with parameter $\nu\in \mathbb{C}^\ast$.

For $p\geq 1$, we define $A^{(p)}=\sum_{w\in
\mathfrak{S}_{p}}T_{w}^\sigma$. The following proposition of
$A^{(p)}$ plays an essential role in the construction of q-trace.
It is due to D. I. Gurevich (\cite{Gu}, Proposition 2.4).
\begin{proposition}For $p\geq1$ we have $$(A^{(p)})^{2}=(p)_\nu!A^{(p)}.$$\end{proposition}

The image of the map $\oplus_{p\geq 0}A^{(p)}$ has important
algebraic structures on it. The first one is the quantum shuffle
product which was introduced by M. Rosso \cite{Ro1, Ro2}. It
generalizes the usual shuffle product on $T(V)$. For any
$v_1,\ldots, v_{i+j}\in V$, the quantum shuffle product $sh$ is
defined to be
\begin{equation*}sh((v_1\otimes\cdots\otimes
v_i)\otimes(v_{i+1}\otimes\cdots\otimes v_{i+j}))=\sum_{w\in
\mathfrak{S}_{i,j}}T_w(v_1\otimes\cdots\otimes
v_{i+j}).\end{equation*} We denote by $T_\sigma(V)$ the quantum
shuffle algebra $(T(V), sh)$. The subalgebra $S_\sigma(V)$ of
$T_\sigma(V)$ generated by $V$ with respect to the quantum shuffle
product is called the \emph{quantum symmetric algebra}. It is easy
to see that $S_\sigma(V)=\oplus_{p\geq 0}\mathrm{Im}(\sum_{w\in
\mathfrak{S}_{p}}T_w)$. We denote by
$S_\sigma^p(V)=\mathrm{Im}(\sum_{w\in \mathfrak{S}_{p}}T_w)$ the
$p$-th component of $S_\sigma(V)$.

The algebra $T_\sigma(V)$ is a coalgebra with the deconcatenation
coproduct $\delta$:
\begin{equation*}\delta(v_1\otimes\cdots\otimes
v_n)=\sum_{i=0}^n(v_1\otimes\cdots\otimes v_i)\otimes
(v_{i+1}\otimes\cdots\otimes v_n).\end{equation*} We denote by
$\delta_{ij}$ the composition of $\delta$ with the projection
$T(V) \otimes T(V) \rightarrow V^{\otimes i}\otimes V^{\otimes
j}$. One can show that $(S_\sigma(V),\delta)$ is also a coalgebra
(\cite{Ro2}).

\subsection{Algebraic structures on $\oplus_{p=0}^{M} \mathrm{End}S_\sigma^p(V)$} Let $\sigma$ be a braiding of Hecke type on $V$ such that
$\dim S_\sigma^M(V)=1$ for some $M$ and $\dim S_\sigma^p(V)=0$ for
$p>M$. For $\textbf{A}\in\oplus_{p=0}^{M}
\mathrm{End}S_\sigma^p(V)$, we write
$\textbf{A}=(A_{0},A_{1},\ldots,A_{M})$, where $A_{p}\in
\mathrm{End}S_\sigma^p(V)$ is the $p$-th component of
$\textbf{A}$.

For $\textbf{A},\textbf{B}\in
\oplus_{p=0}^{M}\mathrm{End}S_\sigma^p(V)$, we define the
\emph{composition product} $\textbf{A}\circ \textbf{B}\in
\oplus_{p=0}^{M}\mathrm{End}S_\sigma^p(V)$ by $(\textbf{A}\circ
\textbf{B})_{p}=A_{p}\circ B_{p}$ with the usual composition.
Obviously, $\oplus_{p=0}^{M}\mathrm{End}S_\sigma^p(V)$ is an
associative algebra with the two-sided unit element
$\textbf{I}=(\texttt{I}_{0},\texttt{I}_{1},\ldots
\texttt{I}_{M})$, where $\texttt{I}_{p}$ is the identity map of
$S_\sigma^p(V)$.

We can also define the \emph{convolution product} $\textbf{A}\ast
\textbf{B}\in \oplus_{p=0}^{M}\mathrm{End}S_\sigma^p(V)$ by
$$(\textbf{A}\ast \textbf{B})_{p}=\sum_{l=0}^{p}A_{l}\ast
B_{p-l},$$ where $A_{i}\ast B_{j}=sh\circ (A_{i}\otimes
B_{j})\circ \delta_{i,j}\in \mathrm{End}S_\sigma^{i+j}(V)$. It is
well-known that the convolution product of endomorphisms is
associative. It follows immediately that
$(\oplus_{p=0}^{M}\mathrm{End}S_\sigma^p(V),\ast)$ is an
associative algebra with the two-sided unit element
$\mathrm{I}_{0}=(\mathrm{I}_{0},0,\ldots,0)$.

\begin{proposition}For $0\leq p\leq M$, we have $$\mathrm{I}_{1}^{\ast p}=(p)_\nu!\mathrm{I}_{p}.$$\end{proposition}
\begin{proof}We first notice that for any $v_1,\ldots, v_p\in V$,$$sh(v_1\otimes sh(v_2\otimes \cdots sh(v_{p-1}\otimes v_p)\cdots))=A^{(p)}(v_1\otimes \cdots \otimes v_p).$$
Then\begin{eqnarray*}\mathrm{I}_{1}^{\ast p}\circ A^{(p)}&=&A^{(p)}\circ\mathrm{I}_{1}^{\otimes p}\circ A^{(p)}\\
&=& (A^{(p)})^2\\
&=&(p)_\nu!A^{(p)}.
\end{eqnarray*}
\end{proof}

\begin{corollary}For $0\leq i,j\leq M$ with $i+j\leq M$, we have $$\mathrm{I}_{i}\ast\mathrm{I}_{j}=\binom{i+j}{i}_\nu \mathrm{I}_{i+j},$$ where $\binom{i+j}{i}_\nu =(i+j)_\nu!/((i)_\nu!(j)_\nu!)$.\end{corollary}

Now we assume that the parameter $\nu$ in the Iwahori's equation
is not a root of unity. For any $A\in
\mathrm{End}S_\sigma^1(V)=\mathrm{End}(V)$, we define $$e^{\ast
A}_\nu=(\texttt{I}_{0}, \frac{1}{(1)_\nu!}A,
\frac{1}{(2)_\nu!}A^{\ast 2},\ldots,\frac{1}{(M)_\nu!}A^{\ast
M}).$$ In particular, $e^{\ast
\texttt{I}_{1}}_\nu=(\texttt{I}_{0},\texttt{I}_{1},\ldots,\texttt{I}_{M}).$

If we write $$(e^{\ast A}_\nu)^{-1}=(\texttt{I}_{0},
\frac{-1}{(1)_\nu!}A, \frac{\nu}{(2)_\nu!}A^{\ast
2},\ldots,\frac{(-1)^{M}\nu^{M(M-1)/2}}{(M)_\nu!}A^{\ast M}),$$
then $$(e^{\ast A}_\nu)^{-1}\ast e^{\ast A}_\nu=e^{\ast A}_\nu\ast
(e^{\ast A}_\nu)^{-1}=\texttt{I}_{0}.$$

We define
\[\begin{array}{cccc}
\alpha: &\oplus_{p=0}^{M}\mathrm{End}S_\sigma^p(V) &\rightarrow&
\oplus_{p=0}^{M}\mathrm{End}S_\sigma^p(V)
,\\[6pt]
&\textbf{A} &\mapsto&\textbf{A} \ast e^{\ast \texttt{I}_{1}}_\nu.
\end{array}\]

Consequently, $\alpha$ has an inverse defined by
$\alpha^{-1}(\textbf{A})=\textbf{A}\ast (e^{\ast
\texttt{I}_{1}}_\nu)^{-1}.$

\begin{definition}For any $\textbf{A},\textbf{B}\in \oplus_{p=0}^{M}\mathrm{End}S_\sigma^p(V)$, the \emph{third product} $\textbf{A}\times \textbf{B}$ of $\textbf{A}$ and $\textbf{B}$ is defined to be $$\textbf{A}\times \textbf{B}=\alpha^{-1}\big((\alpha \textbf{A})\circ(\alpha \textbf{B})\big)=\big((\textbf{A}\ast e^{\ast
\mathrm{I}_{1}}_\nu)\circ(\textbf{B}\ast e^{\ast
\mathrm{I}_{1}}_\nu )\big)\ast( e^{\ast \mathrm{I}_{1}}_\nu)^{-1}
.$$\end{definition}

\begin{proposition}The space
$\oplus_{p=0}^{M}\mathrm{End}S_\sigma^p(V)$ equipped with the
third product is an associative algebra with two-sided unit
element $\texttt{I}_{0}$.
\end{proposition}
\begin{proof}For any $\textbf{A},\textbf{B}, \textbf{C}\in
\oplus_{p=0}^{M}\mathrm{End}S_\sigma^p(V)$, we have
\begin{eqnarray*}
(\textbf{A}\times \textbf{B})\times \textbf{C}&=&\alpha^{-1}\Big(\big(\alpha (\textbf{A}\times \textbf{B})\big)\circ(\alpha \textbf{C})\Big)\\
&=&\alpha^{-1}\Big(\big(\alpha\circ \alpha^{-1}((\alpha \textbf{A})\circ(\alpha \textbf{B}))\big)\circ(\alpha \textbf{C})\Big)\\
&=&\alpha^{-1}\big((\alpha \textbf{A})\circ (\alpha \textbf{B})\circ (\alpha \textbf{C})\big)\\
&=&\textbf{A}\times (\textbf{B}\times \textbf{C}).\\
\end{eqnarray*}
And
\begin{eqnarray*}
\texttt{I}_{0}\times \textbf{A}&=&\alpha^{-1}\big((\alpha \texttt{I}_{0})\circ(\alpha \textbf{A})\big)\\
&=&\alpha^{-1}\big((\texttt{I}_{0}\ast e^{\ast \texttt{I}_{1}}_\nu)\circ(\alpha \textbf{A})\big)\\
&=&\alpha^{-1}\big(e^{\ast \texttt{I}_{1}}_\nu\circ(\alpha \textbf{A})\big)\\
&=&\alpha^{-1}(\alpha \textbf{A})\\
&=&\textbf{A}.
\end{eqnarray*}
Similarly, we have that $\textbf{A}\times
\texttt{I}_{0}=\textbf{A}$.
\end{proof}

\begin{proposition}For $0\leq r\leq M$, $A_{i}\in \mathrm{End}S_\sigma^i(V)$ and $B_{j}\in\mathrm{End}S_\sigma^j(V)$, we have $$(A_{i}\times B_{j})_{r}=\sum^{r}_{s=0}\frac{\nu^{s(s-1)/2}}{(s)_\nu!}\big((A_{i}\ast \mathrm{I}_{r-s-i})\circ (B_{j}\ast \mathrm{I}_{r-s-j})\big)\ast\mathrm{I}_{1}^{\ast s} ,$$ where $\mathrm{I}_{t}=0$ for $t<0$.\end{proposition}
\begin{proof} The formula follows from the definition
of the third product.
\end{proof}

\begin{corollary}We have $(A_{i}\times B_{j})_{r}=0$ for $r< \max(i,j)$ and $(A_{r}\times B_{r})_{r}=A_{r}\circ B_{r}.$ \end{corollary}

\subsection{The q-trace}
\begin{definition}The \emph{q-trace} of any $\textbf{A}\in \oplus_{p=0}^{M}\mathrm{End}S_\sigma^p(V)$ is the unique element $\mathrm{Tr}_{q}\textbf{A} \in \mathbb{C}$ such that $(\alpha \textbf{A})_{M}=(\mathrm{Tr}_{q}\textbf{A})\mathrm{I}_{M}\in \mathrm{End}S_\sigma^M(V)$. \end{definition}

\begin{theorem}The q-trace is an algebra morphism with respect to the third product. Precisely, for $\textbf{A},\textbf{B}\in \oplus_{p=0}^{M}\mathrm{End}S_\sigma^p(V)$, we have

1.
$\mathrm{Tr}_{q}(\textbf{A}+\textbf{B})=\mathrm{Tr}_{q}\textbf{A}+\mathrm{Tr}_{q}\textbf{B},$

2. $\mathrm{Tr}_{q}(\textbf{A}\times
\textbf{B})=(\mathrm{Tr}_{q}\textbf{A})(\mathrm{Tr}_{q}\textbf{B}),$

3. $\mathrm{Tr}_{q}(\textbf{A}\times
\textbf{B})=\mathrm{Tr}_{q}(\textbf{B}\times
\textbf{A}).$\end{theorem}

\begin{proof} 1. By the definition, we have
\begin{eqnarray*}
(\alpha(\textbf{A}+\textbf{B}))_{M}&=&\sum_{k=0}^{M}(A_{k}+B_{k})\ast\texttt{I}_{M-k}\\
&=&(\alpha \textbf{A})_{M}+(\alpha \textbf{B})_{M}.
\end{eqnarray*}
 So
\begin{eqnarray*}
\mathrm{Tr}_{q}(\textbf{A}+\textbf{B})\texttt{I}_{M}&=&(\mathrm{Tr}_{q}\textbf{A})\texttt{I}_{M}+(\mathrm{Tr}_{q}\textbf{B})\texttt{I}_{M}\\
&=&(\mathrm{Tr}_{q}\textbf{A}+\mathrm{Tr}_{q}\textbf{B})\texttt{I}_{M}.
\end{eqnarray*}
Therefore
$\mathrm{Tr}_{q}(\textbf{A}+\textbf{B})=\mathrm{Tr}_{q}\textbf{A}+\mathrm{Tr}_{q}\textbf{B}$.

2. Since $\textbf{A}\times \textbf{B}=\alpha^{-1}((\alpha
\textbf{A})\circ(\alpha \textbf{B}))$, we have
$\alpha(\textbf{A}\times \textbf{B})=(\alpha
\textbf{A})\circ(\alpha \textbf{B})$. So
$$(\alpha(\textbf{A}\times \textbf{B}))_{M}=(\alpha \textbf{A})_{M}\circ(\alpha \textbf{B})_{M},$$
which implies that
\begin{eqnarray*}
\mathrm{Tr}_{q}(\textbf{A}\times \textbf{B})\texttt{I}_{M}&=&(\mathrm{Tr}_{q}\textbf{A})\texttt{I}_{M}\circ (\mathrm{Tr}_{q}\textbf{B})\texttt{I}_{M}\\
&=&(\mathrm{Tr}_{q}\textbf{A})
(\mathrm{Tr}_{q}\textbf{B})\texttt{I}_{M}.
\end{eqnarray*}
So we have $\mathrm{Tr}_{q}(\textbf{A}\times
\textbf{B})=(\mathrm{Tr}_{q}\textbf{A})(\mathrm{Tr}_{q}\textbf{B})$.

3. It follows from the identity stated in 2 immediately.
\end{proof}

\section{Another approach of the quantum trace}
In the previous section we have defined the third product and the
q-trace in a general setting. In this section, we study a special
case of braided vector spaces which provides an elementary
approach to the quantum trace of type $A$. We first introduce the
quantum exterior algebra which is the quantum symmetric algebra
related to the fundamental representation of
$\mathcal{U}_{q}\mathfrak{sl}_{N+1}$. And then we give a more
explicit law for the second product in this case. Using the
computational result, we give the formula of the q-trace and
compare it with the quantum trace of type $A$.

\subsection{Quantum exterior algebras}
In the rest of this paper, we denote $V=\mathbb{C}^{N+1}$ and by
$E_{ij}$ the matrix with entry 1 in the position $(i,j)$ and
entries 0 elsewhere. The \emph{fundamental representation} of
$\mathcal{U}_{q}\mathfrak{sl}_{N+1}$ is the algebra homomorphism
\[\begin{array}{cccc}
\rho: &\mathcal{U}_{q}\mathfrak{sl}_{N+1}&\rightarrow &\mathrm{End}V,\\
 &E_{i} &\mapsto &E_{i,i+1},\\
 &F_{i}&\mapsto & E_{i+1,i},\\
&K_{i}& \mapsto &\sum_{l\neq i,
i+1}E_{ll}+qE_{ii}+q^{-1}E_{i+1,i+1},
\end{array}\]
where $E_i$'s, $F_i$'s and $K_i$'s are the standard generators of
$\mathcal{U}_{q}\mathfrak{sl}_{N+1}$.

Then the action of the $R$-matrix on $V\otimes V$is given by
$$R_{\rho}=q\sum_{i=1}^{N+1}E_{ii}\otimes E_{ii}+\sum_{i\neq j}E_{ij}\otimes E_{ji}+(q-q^{-1})\sum_{i<j}E_{jj}\otimes E_{ii}.$$

Let $c=q^{-1}R_{\rho}\in \mathrm{GL}(V\otimes V)$. If we denote by
$e_{i}=(0,\ldots,0,1,0,\ldots,0)^{t}\in V$ the unit column vector
whose components are zero except the i-th component is 1, then we
have:
\[c(e_{i}\otimes e_{j})=\left\{
\begin{array}{lll}
e_{i}\otimes e_{i},&& i=j,\\
q^{-1}e_{j}\otimes e_{i},&& i<j,\\
q^{-1}e_{j}\otimes e_{i}+(1-q^{-2})e_{i}\otimes e_{j},&&i>j.
\end{array} \right.
\]

The map $c$ is a braiding on $V$ and satisfies the \emph{Iwahori's
quadratic equation}:
\begin{equation*}(c-\mathrm{id}_{V\otimes V})(c+q^{-2}\mathrm{id}_{V\otimes
V})=0.\end{equation*}

\begin{definition}Let $\mathfrak{I}$ be the two-sided ideal generated by
$\mathrm{Ker}(\mathrm{id}_{V\otimes V}-c)$ in $\mathrm{T}(V)$. The
quotient algebra
$\bigwedge\nolimits_c(V)=\mathrm{T}(V)/\mathfrak{I}$ is called the
\emph{quantum exterior algebra} on $V$.\end{definition}

By an easy computation, we have
\begin{equation*}\mathrm{Ker}(\mathrm{id}_{V\otimes
V}-c)=\mathrm{Span}_{\mathbb{C}}\{ e_{i}\otimes
e_{i},q^{-1}e_{i}\otimes e_{j}+e_{j}\otimes
e_{i}(i<j)\}.\end{equation*} Let $\pi:\mathrm{T}(V)\rightarrow
\bigwedge_{c}(V)$ be the canonical projection. For any
$e_{i_{1}}\otimes \cdots \otimes e_{i_{p}}\in \mathrm{T}^{p}(V)$,
we denote $e_{i_{1}}\wedge \cdots \wedge
e_{i_{p}}=\pi(e_{i_{1}}\otimes \cdots \otimes e_{i_{p}})$. It
follows immediately that:

1. The algebra $\bigwedge_{c}(V)$ is graded and generated by
$\{e_{1},\ldots , e_{N+1}\}$ with the relations:
$$e_{i}\wedge e_{i}=0,$$ and $$e_{j}\wedge e_{i}=-q^{-1}e_{i}\wedge e_{j}\ (i<j).$$

2. If we denote by $\bigwedge_{c}^{p}(V)$ the $p$-th component of
$\bigwedge_{c}(V)$, then $\dim\bigwedge_{c}^{N+1}(V)=1$ and
$\bigwedge_{c}^{p}(V)=0$ for $p>N+1$.

3. The set $\{e_{i_{1}}\wedge \cdots \wedge e_{i_{p}}|1\leq i_{1}
< \cdots < i_{p}\leq N+1,1\leq p\leq N+1\}$ forms a linear basis
of $\bigwedge_{c}(V)$.\\

Let $A^{(p)}=\sum_{w\in \mathfrak{S}_{p}}T_{w}^{-c}$. Then by the
following proposition, we can view the quantum exterior algebra as
a special quantum symmetric algebra.

\begin{proposition}[\cite{Gu}, Proposition 2.13]For $k\geq 1$, we have the following linear isomorphism:
$$\mathrm{Im}A^{(k)}\cong \wedge^{k}_{c}(V).$$\end{proposition}

We can identify $\bigwedge_{c}(V)$ with $S_{-c}(V)$ as linear
space. Moreover, since $sh(A^{(i)}\otimes A^{(j)})=\sum_{w\in
\mathfrak{S}_{i,j}}T_{w}^{-c}(A^{(i)}\otimes A^{(j)})=A^{{i+j}}$,
the product in $\bigwedge_{c}(V)$ is just the quantum shuffle
product in $S_{-c}(V)$. The space $\bigwedge_{c}(V)$ also inherits
the coproduct of $S_{-c}(V)$. It is not difficult to show the
following formula for the deconcatenation coproduct on
$\bigwedge_{c}(V)$: for $1\leq t\leq p\leq N+1$ and $1\leq
i_{1}<i_{2}<\cdots <i_{p}\leq N+1$,
$$\delta_{t,p-t}(e_{i_{1}}\wedge \cdots
\wedge e_{i_{p}})=\sum_{w\in
\mathfrak{S}_{t,p-t}}(-q)^{-l(w)}e_{i_{w(1)}}\wedge  \cdots \wedge
e_{i_{w(t)}}\otimes e_{i_{w(t+1)}}\wedge \cdots \wedge
e_{i_{w(p)}}.$$

\subsection{Explicit law of the second product}
In order to give an explicit formula of the q-trace, we describe
the convolution product more precisely in our special case.

 We define $c^{\vee}=(c^{-1})^{t}$,
where $t$ means the transpose of the operator. Then $c^{\vee}\in
\mathrm{GL}(V^{\ast}\otimes V^{\ast})$. Let $\{f_{i}\}$ be the
dual basis of $\{e_{i}\}$. We have
\[c^{\vee}(f_{i}\otimes f_{j})=\left\{
\begin{array}{lll}
f_{i}\otimes f_{i},&& i=j,\\
qf_{j}\otimes
f_{i}+(1-q^{2})f_{i}\otimes f_{j},&&i<j,\\
qf_{j}\otimes f_{i},&& i>j.
\end{array} \right.
\]
Obviously $c^{\vee}$ is a braiding on $V^{\ast}$ and satisfies the
Iwahori's equation:
\begin{equation*}(c^{\vee}-\mathrm{id}_{V^\ast\otimes
V^\ast})(c^{\vee}+q^{2}\mathrm{id}_{V^\ast\otimes
V^\ast})=0.\end{equation*}

It is easy to show that $\bigwedge\nolimits_{c^{\vee}}(V^{\ast})$,
as an algebra, is generated by $f_{i} $'s with the
relations:$$f_{i}\wedge f_{i}=0,\ f_{j}\wedge
f_{i}=-q^{-1}f_{i}\wedge f_{j}\ (i<j).$$ Therefore the map $e_{i}
\mapsto f_{i} $ induces an isomorphism of algebras:
$\bigwedge\nolimits_{c}(V)\rightarrow
\bigwedge\nolimits_{c^{\vee}}(V^{\ast})$.

For any $s<t$, we have
\begin{eqnarray*}
E_{ij}\ast E_{kl}(e_{s}\wedge e_{t})&=&sh\circ (E_{ij}\otimes E_{kl})\circ \delta_{1,1}(e_{s}\wedge e_{t})\\
&=&sh\circ (E_{ij}\otimes E_{kl})(e_{s}\otimes e_{t}-q^{-1}e_{t}\otimes e_{s})\\
&=&\delta _{js}\delta _{lt}e_{i}\wedge e_{k}-q^{-1}\delta _{jt}\delta _{ls}e_{i}\wedge e_{k}\\
&=&(\delta _{js}\delta _{lt}-q^{-1}\delta _{jt}\delta _{ls})e_{i}\wedge e_{k}.\\
\end{eqnarray*}
Similarly, $E_{kl}\ast E_{ij}(e_{s}\wedge e_{t})=(\delta
_{ls}\delta _{jt}-q^{-1}\delta _{lt}\delta _{js})e_{k}\wedge
e_{i}$. So we get that
\[\left\{
\begin{array}{lllll}
E_{ij}\ast E_{ik}&=&E_{ij}\ast E_{kj}=0,&& \forall i,j,k,\\
E_{kj}\ast E_{il}&=&-q^{-1}E_{ij}\ast E_{kl},&& \mathrm{if} i<k,\forall j,l,\\
E_{il}\ast E_{kj}&=&-q^{-1}E_{ij}\ast E_{kl},&& \mathrm{if}
j<l,\forall i,k.
\end{array} \right. \eqno(1)
\]
In general, for $1\leq i_{1}<\cdots <i_{p}\leq N+1$, $1\leq
j_{1}<\cdots <j_{p}\leq N+1$ and $1\leq l_{1}<\cdots <l_{p}\leq
N+1$, we have
\[ E_{i_{1}j_{1}}\ast \cdots \ast E_{i_{p}j_{p}}(e_{l_{1}}\wedge \cdots\wedge e_{l_{p}})=\left\{
\begin{array}{lll}
e_{i_{1}}\wedge \cdots\wedge
e_{i_{p}},&&\mathrm{if}j_k=l_k,\\[3pt]
0,&& \mathrm{otherwise}.
\end{array} \right.
\]
So the set $$\{E_{i_{1}j_{1}}\ast \cdots \ast E_{i_{p}j_{p}}|1\leq
i_{1}<\cdots <i_{p}\leq N+1,1\leq j_{1}<\cdots <j_{p}\leq N+1\}$$
forms a linear basis of
$\mathrm{End}\bigwedge\nolimits_{c}^{p}(V)$, which implies that
$\oplus_{p=0}^{N+1}\mathrm{End}\bigwedge_{c}^{p}(V)$ is an algebra
generated by $\{E_{ij}\}$. As a consequence, we have the following
linear isomorphism:
\[\begin{array}{cccc}
\iota_{p}:&\mathrm{End}\bigwedge\nolimits_{c}^{p}(V)&\rightarrow&\bigwedge\nolimits_{c}^{p}(V)\otimes \bigwedge\nolimits_{c^{\vee}}^{p}(V^{\ast}) ,\\
 &E_{i_{1}j_{1}}\ast \cdots \ast E_{i_{p}j_{p}} &\mapsto & e_{i_{1}}\wedge
\cdots \wedge e_{i_{p}}\otimes f_{j_{1}}\wedge \cdots \wedge
f_{j_{p}}.
\end{array}\]

If we endow $\bigwedge_{c}(V)\otimes
\bigwedge_{c^{\vee}}(V^{\ast})$ with the tensor algebra structure,
then $\oplus_{p=0}^{N+1}\bigwedge_{c}^{p}(V)\otimes
\bigwedge_{c^{\vee}}^{p}(V^{\ast})$ is a subalgebra.

\begin{proposition}The map $$\iota=\oplus_{p=0}^{N+1}\iota_{p}: (\oplus_{p=0}^{N+1}\mathrm{End}\bigwedge\nolimits_{c}^{p}(V),\ast)\rightarrow\oplus_{p=0}^{N+1}\bigwedge\nolimits_{c}^{p}(V)\otimes \bigwedge\nolimits_{c^{\vee}}^{p}(V^{\ast})$$ is an isomorphism of algebras. \end{proposition}
\begin{proof}In $\oplus_{p=0}^{N+1}\bigwedge_{c}^{p}(V)\otimes
\bigwedge_{c^{\vee}}^{p}(V^{\ast})$, we have
\[\left\{
\begin{array}{lllll}
(e_{i}\otimes f_{j})(e_{i}\otimes f_{k})&=&0,&& \forall i,j,\\
(e_{i}\otimes f_{j})(e_{k}\otimes f_{j})&=&0,&& \forall i,j,\\
(e_{k}\otimes f_{j})(e_{i}\otimes f_{l})&=&-q^{-1}(e_{i}\otimes f_{j})(e_{k}\otimes f_{l}),&& \mathrm{if} i<k,\forall j,l,\\
(e_{i}\otimes f_{l})(e_{k}\otimes f_{j})&=&-q^{-1}(e_{i}\otimes f_{j})(e_{k}\otimes f_{l}),&& \mathrm{if} j<l,\forall i,k.\\
\end{array} \right.
\]
It shares the same multiplication rule in (1). And
$\oplus_{p=0}^{N+1}\bigwedge_{c}^{p}(V)\otimes
\bigwedge_{c^{\vee}}^{p}(V^{\ast})$ is generated by $e_{i}\otimes
f_{j}$'s as an algebra. So we get the conclusion.  \end{proof}

\subsection{More information about the q-trace}
Using the formula (1), we compute the q-trace on
$\oplus_{p=0}^{N+1}\mathrm{End}\bigwedge_{c}^{p}(V)$. We also give
an inductive formula of the q-trace.

\begin{proposition}For any $A\in \mathrm{End}\bigwedge_{c}^{p}(V)$ with $A=\sum_{\substack{1\leq i_{1}<\cdots <i_{p}\leq N+1\\1\leq
j_{1}<\cdots <j_{p}\leq N+1}}a^{i_{1} \cdots
 i_{p}}_{j_{1} \cdots j_{p}}E_{i_{1}j_{1}}\ast \cdots \ast E_{i_{p}j_{p}}$, we have
 \begin{eqnarray*}\mathrm{Tr}_q A&=&\sum_{w\in \mathfrak{S}_{p,\
N+1-p}}(-q)^{-2l(w)}a^{w(1) \cdots
 w(p)}_{w(1)\cdots w(p)}\\
 &=&\sum_{1\leq l_{1}<\cdots <l_{p}\leq N+1}q^{(p+1)p-2(l_1+\cdots+l_p)}a^{l_{1} \cdots
 l_{p}}_{l_{1} \cdots l_{p}}.\end{eqnarray*}In particular,  If $A\in
\mathrm{End}\bigwedge_{c}^{1}(V)=\mathrm{End}(V)$ with
 $A=\sum
a^{j}_{i}E_{ji}$,
then$$\mathrm{Tr}_{q}A=\sum^{N+1}_{i=1}q^{-2(i-1)}a^{i}_{i}.$$\end{proposition}
\begin{proof}According to the definition, we have
\begin{eqnarray*}\lefteqn{A\ast \mathrm{I}_{N+1-p}}\\
&=&(\sum_{\substack{1\leq i_{1}<\cdots <i_{p}\leq N+1\\1\leq
j_{1}<\cdots <j_{p}\leq N+1}}a^{i_{1} \cdots
 i_{p}}_{j_{1} \cdots j_{p}}E_{i_{1}j_{1}}\ast \cdots \ast E_{i_{p}j_{p}})\\[3pt]
&&\ast(\sum_{1\leq k_1<\cdots<k_{N+1-p}\leq N+1}E_{k_{1}k_{1}}\ast \cdots \ast E_{k_{N+1-p}k_{N+1-p}})\\[3pt]
&=&\sum_{\substack{1\leq i_{1}<\cdots <i_{p}\leq N+1\\1\leq
j_{1}<\cdots <j_{p}\leq N+1\\1\leq k_1<\cdots<k_{N+1-p}\leq N+1}}
a^{i_{1} \cdots
 i_{p}}_{j_{1} \cdots j_{p}}E_{i_{1}j_{1}}\ast \cdots \ast E_{i_{p}j_{p}}\ast E_{k_{1}k_{1}}\ast \cdots \ast E_{k_{N+1-p}k_{N+1-p}}\\
&=&\sum_{w\in \mathfrak{S}_{p,\ N+1-p}}a^{w(1) \cdots
 w(p)}_{w(1)\cdots w(p)}E_{w(1)w(1)}\ast \cdots \ast E_{w(N+1)w(N+1)}\\
&=&\sum_{w\in \mathfrak{S}_{p,\ N+1-p}}(-q)^{-2l(w)}a^{w(1) \cdots
 w(p)}_{w(1)\cdots w(p)}E_{11}\ast \cdots \ast E_{N+1,N+1}\\
 &=&(\sum_{w\in \mathfrak{S}_{p,\ N+1-p}}(-q)^{-2l(w)}a^{w(1) \cdots
 w(p)}_{w(1)\cdots w(p)})\mathrm{I}_{N+1}.
\end{eqnarray*}
From an easy observation, we know that for any $w\in
\mathfrak{S}_{p,N+1-p}$ with $w(1)=l_1,\cdots,w(p)=l_p$ we have
$l(w)=(l_1-1)+\cdots+(l_p-p)=(l_1+\cdots+l_p)-(1+p)p/2$.
\end{proof}

\begin{proposition}For any $A\in \mathrm{End}(V)$ with $Ae_{i}=\sum^{N+1}_{j=1}a^{j}_{i}e_{j}$,and $0\leq p \leq N+1$, we have
$$\mathrm{Tr}_{q}A^{p}=\sum^{n}_{i=1}q^{-2(i-1)}\sum_{j_{1},\cdots,
j_{p}}a^{i}_{j_{1}}a^{j_{1}}_{j_{2}}\cdots a^{j_{p}}_{i},$$ and
$$\mathrm{Tr}_{q}A^{\ast p}=\sum_{\theta ,\tau\in \mathfrak{S}_{p}}\sum_{w\in
\mathfrak{S}_{p,N+1-p}}(-q)^{-(2l(w)+l(\theta)+l(\tau))}a^{\tau
w(1)}_{\theta w(1)}\cdots a^{\tau w(p)}_{\theta w(p)}.$$

In particular,

$$\mathrm{Tr}_{q}A^{\ast N+1}=\sum_{\theta ,\tau\in
\mathfrak{S}_{N+1}}(-q)^{-l(\theta)-l(\tau)}a^{\tau (1)}_{\theta
(1)}\cdots a^{\tau (N+1)}_{\theta (N+1)}.$$
\end{proposition}
\begin{proof}All identities follows from direct computation.
\end{proof}

For a diagonalizable $A\in \mathrm{End}V$ with
$Ae_{i}=a^{i}_{i}e_{i}$, we have
\[\left\{
\begin{array}{lll}
\mathrm{Tr}_{q}A^{\ast
N+1}&=&(N+1)_{q^{-2}}!a^{1}_{1}\cdots a^{N+1}_{N+1},\\[3pt]
\mathrm{Tr}_{q}A^{\ast p}&=&(p)_{q^{-2}}!\sum_{w\in
\mathfrak{S}_{p,N+1-p}}(-q)^{-2l(w)}a^{ w(1)}_{ w(1)}\cdots a^{
w(p)}_{w(p)},\\[3pt]
\mathrm{Tr}_{q}A^{p}&=&\sum_{i=0}^{N+1}(-q)^{-2(i-1)}(a^{i}_{i})^{p}.
\end{array} \right.
\]

Let $\mathbb{C}=V(1)\subset V(2)\subset \cdots \subset V(i)\subset
\cdots$ be a sequence of vector spaces with $
V(i)=Span_\mathbb{C}\{e_1,\ldots,e_i\}$. We still use $c$ to
denote the action of $c$ restricted on $V(i)$ for all $i$. For any
$1\leq p\leq N$, we define
\[\begin{array}{cccc}
(\mathrm{Tr}_q)_{p+1}:& \mathrm{End}\bigwedge_{c}^{p+1}(V_{N+1})
&\rightarrow& \mathrm{End}\bigwedge\nolimits_{c}^{p}(V_N)
,\\[6pt]
&E_{i_{1}j_{1}}\ast \cdots \ast E_{i_{p+1}j_{p+1}}
&\mapsto&E_{i_{1}j_{1}}\ast \cdots \ast E_{i_p j_p}\mathrm{Tr}_q
E_{i_{p+1}j_{p+1}},
\end{array}\]
where $1\leq i_{1}<\cdots <i_{p+1}\leq N+1$ and $1\leq
j_{1}<\cdots <j_{p+1}\leq N+1$.
\begin{proposition}For any $A\in \mathrm{End}\bigwedge_{c}^{p}(V)$, we have $$\mathrm{Tr}_q A=(-q)^{p(p-1)}(\mathrm{Tr}_q)_{1}(\mathrm{Tr}_q)_{2}\cdots(\mathrm{Tr}_q)_{p}A.$$
\end{proposition}
\begin{proof}We set $A=\sum_{\substack{1\leq i_{1}<\cdots <i_{p}\leq N+1\\1\leq
j_{1}<\cdots <j_{p}\leq N+1}}a^{i_{1} \cdots
 i_{p}}_{j_{1} \cdots j_{p}}E_{i_{1}j_{1}}\ast \cdots \ast E_{i_{p}j_{p}}$. Then
\begin{eqnarray*}\lefteqn{(\mathrm{Tr}_q)_{1}(\mathrm{Tr}_q)_{2}\cdots(\mathrm{Tr}_q)_{p}A}\\
&=&\sum_{\substack{1\leq i_{1}<\cdots <i_{p}\leq N+1\\1\leq
j_{1}<\cdots <j_{p}\leq N+1}}a^{i_{1} \cdots
 i_{p}}_{j_{1} \cdots j_{p}}\mathrm{Tr}_q E_{i_{1}j_{1}}\cdot \cdots \cdot \mathrm{Tr}_q E_{i_{p}j_{p}}\\
&=&\sum_{\substack{1\leq i_{1}<\cdots <i_{p}\leq N+1\\1\leq
j_{1}<\cdots <j_{p}\leq N+1}}a^{i_{1} \cdots
 i_{p}}_{j_{1} \cdots j_{p}}\delta_{i_{1}j_{1}}(-q)^{-2(i_1 -1)}\cdot \cdots \cdot \delta_{i_{p}j_{p}}(-q)^{-2(i_1 -1)}\\
&=&\sum_{1\leq i_{1}<\cdots <i_{p}\leq N+1}a^{i_{1} \cdots
 i_{p}}_{i_{1} \cdots i_{p}}(-q)^{-2(i_1+\cdots+i_p-p)}\\
&=&(-q)^{p(1-p)}\mathrm{Tr}_q A.
\end{eqnarray*}
\end{proof}
\subsection{The relation between q-traces and quantum traces}

Now we recall the definition of the quantum trace. For more
information, one can see \cite{KRT}.

 We know the positive roots of
$\mathfrak{sl}_{N+1}(\mathbb{C})$ are

\[ \begin{array}{ccccc}
\alpha_{1},&
\alpha_{1}+\alpha_{2},&\alpha_{1}+\alpha_{2}+\alpha_{3},&\ldots,&\alpha_{1}+\cdots+\alpha_{N},\\
\alpha_{2},&
\alpha_{2}+\alpha_{3},&\alpha_{2}+\alpha_{3}+\alpha_{4},&\ldots,&\alpha_{2}+\cdots+\alpha_{N},\\
\ldots,&&&&\\
\alpha_{N}.&&&&
\end{array}.
\]

The sum of all positive roots is
$$\sum_{i=1}^{N}i(N+1-i)\alpha_{i}.$$

Set $$K=K_{1}^{N}K_{2}^{2(N-1)}\cdots K_{N}^{N}.$$

For any $A\in \mathrm{End}(V)$, we call
$$\mathrm{tr}_{q}(A)=\mathrm{Tr}(\rho(K)A)$$
the \emph{quantum trace} of $A$, where $\mathrm{Tr}$ is the usual
trace of endomorphisms. By direct computation, one gets that if
$A\in \mathrm{End}(V)$ with
$Ae_{i}=\sum^{N+1}_{j=1}a^{j}_{i}e_{j}$,
then$$\mathrm{tr}_{q}(A)=\sum^{N+1}_{i=1}q^{N-2(i-1)}a^{i}_{i}.$$

Hence, we get that:
\begin{theorem}For any $A\in \mathrm{End}(V)$, we have $$\mathrm{Tr}_{q}A=q^{-N}\mathrm{tr}_{q}(A).$$\end{theorem}

In general, the quantum trace $\mathrm{tr}_q A$ for $A\in
\mathrm{End}\bigwedge_{c}^{p}(V)$ is defined by: $$\mathrm{tr}_q
A=\mathrm{tr}(\rho^p(K)A),$$ where $\rho^p:
\mathcal{U}_{q}\mathfrak{sl}_{N+1}\rightarrow
\mathrm{End}\bigwedge_{c}^{p}(V)$ is the representation of
$\mathcal{U}_{q}\mathfrak{sl}_{N+1}$ on $\bigwedge_{c}^{p}(V)$
induced by the fundamental representation $\rho$. For $1\leq
j_{1}<\cdots <j_{p}\leq N+1$, we have
\begin{eqnarray*}\lefteqn{\rho^p(K)A(e_{j_1}\wedge\cdots\wedge e_{j_p})}\\
&=&\rho^p(K)(\sum_{1\leq i_{1}<\cdots <i_{p}\leq N+1}a^{i_{1}
\cdots
 i_{p}}_{j_{1} \cdots j_{p}}e_{i_1}\wedge\cdots\wedge e_{i_p})\\
&=&\sum_{1\leq i_{1}<\cdots <i_{p}\leq N+1}a^{i_{1} \cdots
 i_{p}}_{j_{1} \cdots j_{p}}Ke_{i_1}\wedge\cdots\wedge Ke_{i_p}\\
&=&\sum_{1\leq i_{1}<\cdots <i_{p}\leq N+1}a^{i_{1} \cdots
 i_{p}}_{j_{1} \cdots j_{p}}q^{p(N+2)-2(i_1+\cdots+i_p)}e_{i_1}\wedge\cdots\wedge e_{i_p},
\end{eqnarray*}
where the last equality follows from
$K=\mathrm{diag}(q^N,q^{N-2},\cdots,q^{-N})$.

So
$$\mathrm{tr}_q A=\sum_{1\leq i_{1}<\cdots <i_{p}\leq
N+1}q^{p(N+2)-2(i_1+\cdots+i_p)}a^{i_{1} \cdots
 i_{p}}_{i_{1} \cdots i_{p}}.$$

Therefore we have the generalization of the above theorem:
\begin{theorem}For any $A\in \mathrm{End}\bigwedge_{c}^{p}(V)$, we have $$\mathrm{Tr}_q A=q^{-p(N+1-p)}\mathrm{tr}_q A.$$
\end{theorem}

\section*{Acknowledgements}This work was
partially supported by China-France Mathematics Collaboration
Grant 34000-3275100 from Sun Yat-sen University. The author would
like to thank Professor Marc Rosso sincerely from whom he got the
leading idea of this paper and useful discussions. He would like
to thank le DMA de l'ENS de Paris for supplying him a very
excellent working environment when he prepared this work. He would
like to thank the referee for careful reading and useful comments
which improved the clarity of the exposition.


\begin{thebibliography}{00}
\bibitem[1]{D}V. G. Drinfel'd, Quantum groups, Proc. Int.
Cong. Math., Berkeley, (1986) pp. 798-820.

\bibitem[2]{FG}D. Flores de Chela, J. A. Green, Quantum symmetric algebras
II, J.Algebra 269 (2003) pp. 610-631.

\bibitem[3]{Gu}D. I. Gurevich, Algebraic aspects of the quantum
Yang-Baxter equation. (Russian) Algebra i Analiz 2 (1990) pp.
119--148; translation in Leningrad Math. J. 2 (1991) pp. 801--828.

\bibitem[4]{HH}M. Hashimoto, T. Hayashi, Quantum
multilinear algebra, T$\hat{o}$hoku Math. J. 44 (1992) pp.
471-521.

\bibitem[5]{J}M. Jimbo, A $q$-difference analogue of $U(\mathfrak{g})$ and the Yang-Baxter equation,  Lett. Math. Phys. 10 (1985) pp. 63--69.

\bibitem[6]{Ka}C. Kassel, Quantum groups, Graduate Texts in Mathematics 155, Springer-Verlag, New York,
1995.

\bibitem[7]{KRT}C. Kassel, M. Rosso, V. Turaev, Quantum groups
and knot invariants, Panoramas et Synth\`{e}ses, num\'{e}ro 5,
Soci\'{e}t\'{e} Math\'{e}matique de France, pp. 1997.

\bibitem[8]{O1}H. Osborn, The Chern-Weil construction, Differential geometry (Proc. Sympos. Pure Math., Vol. XXVII, Stanford Univ., Stanford, Calif., 1973), Part 1, Amer. Math. Soc. (1975) pp. 383--395.

\bibitem[9]{O2}H. Osborn, The trace as an algebra homomorphism, Enseign. Math. (2) 31 (1985) pp. 213--225.

\bibitem[10]{Ro1}M. Rosso, Groupes quantiques et alg\`{e}bres de battage quantiques, C. R. Acad. Sci. Paris I 320 (1995), pp. 145--148.

\bibitem[11]{Ro2}M. Rosso, Quantum groups and quantum shuffles,
Invent. Math. 133 (1998) pp. 399-416.

\bibitem[12]{Wam}M. Wambst, Complexes de Koszul quantiques, Ann. Inst. Fourier (Grenoble) 43 (1993) pp. 1089--1156.




\end{thebibliography}
\end{document}